\documentclass[11pt]{article}

\usepackage[a4paper]{geometry}
\usepackage{amssymb}
\renewcommand{\geq}{\geqslant}
\renewcommand{\leq}{\leqslant}

\usepackage{tikz}
\usepackage{amsthm,amsmath, amssymb}
\usepackage{mathtools}
\usepackage{mathrsfs}
\usepackage{rotating}
\usepackage{mathdots}

\theoremstyle{plain}
\newtheorem{theorem}{Theorem}
\newtheorem{lemma}[theorem]{Lemma}
\newtheorem{corollary}[theorem]{Corollary}
\newtheorem{proposition}[theorem]{Proposition}

\theoremstyle{definition}
\newtheorem{definition}[theorem]{Definition}
\newtheorem{example}[theorem]{Example}

\newtheorem{construction}[theorem]{Construction}

\theoremstyle{remark}

\newcommand{\h}{i_H}
\newcommand{\MU}{\check{\mu}}

\usetikzlibrary{arrows,matrix}
\newcommand{\junk}[1]{}

\newcommand{\set}[1]{\left\{#1\right\}}

\newcommand{\abs}[1]{\lvert {#1} \rvert }

\newcommand{\smin}{\setminus}

\newcommand{\rv}{\overleftarrow}

\renewcommand{\phi}{\varphi}

\newcommand{\calN}{\mathscr{N}}
\newcommand{\calM}{\mathscr{M}}

\newcommand{\rot}[2] {\begin{turn}{#1} #2 \end{turn}}

\newcommand{ \ch}{\color{black}}
\title{The Hamiltonian problem and $t$-path traceable graphs}

\author{Kashif Bari\\
\small Department of Mathematics and Statistics\\[-0.8ex]
\small San Diego State University\\[-0.8ex]
\small San Diego, California\\
\small\tt kashbari@math.tamu.edu\\
\and
Michael E. O'Sullivan \\
\small Department of Mathematics and Statistics\\[-0.8ex]
\small San Diego State University\\[-0.8ex]
\small San Diego, California\\
\small\tt mosullivan@mail.sdsu.edu
}

\date{}

\begin{document}
\maketitle

\begin{abstract}
The problem of characterizing maximal non-Hamiltonian graphs may be
naturally extended to characterizing graphs that are maximal with
respect to non-traceability and beyond that to $t$-path traceability.
We show how traceability behaves with respect to
disjoint union of graphs and the join with a complete graph.  Our main
result is  a decomposition theorem that reduces the problem of characterizing maximal
$t$-path traceable graphs to characterizing those that have no universal vertex.
We generalize a construction of maximal non-traceable graphs by Zelinka to $t$-path traceable graphs.

\end{abstract}

\section{Introduction}
The motivating problem for this article is the characterization of maximal non-Hamiltonian (MNH)
graphs.   
Skupien and co-authors give the first broad family of MNH graphs in~\cite{SkupienMNH}
and describe all MNH graphs with 10 or fewer vertices in ~\cite{SkupienCat}.  The latter paper also includes
three constructions---types $A1$, $A2$, $A3$---with a similar structure.  
Zelinka gave two constructions of graphs that are maximal non-traceable; that is, they have
no Hamiltonian path, but the addition of any edge gives a Hamiltonian path.
The join of such a graph with a single vertex gives a MNH graph.  Zelinka's first family
produces, under the join with $K_1$, the  Skupien MNH graphs from
\cite{SkupienMNH}.
Zelinka's second family is a broad generalization of the type $A1$,
$A2$, and $A3$ graphs of \cite{SkupienCat}.
Bullock et al~\cite{Bullock} provide further examples of  infinite families of maximal
non-traceable graphs. 

In this article we work with two closely related invariants of a graph~$G$, $\MU(G)$ and $\mu(G)$.
The $\mu$-invariant, introduced by Ore \cite{Ore}, is the maximal number of paths in $G$ required to cover
the vertex set of $G$.  We show that $\MU(G)= \mu(G)$ unless $G$ is Hamiltonian, when $\MU(G)=0$.
Maximal non-Hamiltonian graphs are maximal with respect to $\MU(G)=1$, and maximal non-traceable
graphs are maximal with respect to $\MU(G)=2$.  It is useful to broaden the perspective
to study, for arbitrary $t$,  graphs that are  maximal with respect to
$\MU(G) = t$, which we call $t$-path traceable graphs.  

In Section~\ref{s:traceability} we show how the $\MU$ and $\mu$ invariants behave with respect to
disjoint union of graphs and the join with a complete graph.  Section~\ref{s:decomposition} derives
the main result, a decomposition theorem that reduces the problem of characterizing maximal
$t$-path traceable to characterizing those that have no universal vertex,
which we call trim.  Section~\ref{s:family} presents a
generalization of the Zelinka construction to $t$-path traceable graphs.

\section{Traceability and Hamiltonicity}
\label{s:traceability}


It will be notationally convenient to say that the complete graphs $K_1$ and $K_2$ are Hamiltonian.  
As justification for this view, consider an undirected graph as
a directed graph with each edge having a conjugate edge in the reverse
direction.  This perspective does not affect the Hamiltonicity of a
graph with more than 3 vertices, but it does give  $K_2$ a Hamiltonian cycle.   
Similarly, adding loops to any graph with more than~2 vertices
does not alter the Hamiltonicity of the graph, but $K_1$, with
an added loop, has a Hamiltonian cycle.

Let $G$ be a graph.  
A vertex, $v \in V(G)$ , is called a {\em universal vertex} if $\deg(v) = |V(G)|-1$.
Let $\overline{G}$ denote the {\em graph complement} of $G$, having
vertex set $V(G)$ and edge set $E(K_n) \smin E(G)$.
We will use the disjoint union of two graphs, $G \sqcup H$  and the
join of two graphs  $G\ast H$.  The latter  is  $G \sqcup H $ together with
the  edges $\set{vw \vert  v \in V(G) \text{ and } w \in V(H)}$.

\begin{definition}
 A set of $s$ disjoint paths in a graph $G$ that includes every
vertex in $G$ is a {\em $s$-path covering} of $G$.  Define the following invariants.

 $\displaystyle  \mu(G) := \min_{s\in \mathbb{N}} \{ \exists s \text{-path covering of } G \}$.  

$\displaystyle \MU(G) := \min_{l\in \mathbb{N}_0} \{ K_l  \ast G \text{ is Hamiltonian } \}$

$\displaystyle \h(G) := 
\begin{cases}
       1 &  \text{ if } G \text{ is Hamiltonian}\\
       0 & \text{ otherwise} 
\end{cases}$

We will say $G$ is {\em $t$-path traceable} when $\mu(G) = t$.  
A set of $t$ disjoint paths that cover a $t$-path traceable graph $G$ is a {\em minimal path  covering}.  
\end{definition}

Note that $K_r*(K_s*G) = K_{r+s}*G$.  
If $G$ is  Hamiltonian then so is $K_r*G$ for $r\geq 0$. (In particular
this is true for $G = K_1$ and $G=K_2$.)

We now have a series of lemmas that lead to the main result of this section, which is a formula showing how
the $\mu$-invariant and $\MU$-invariant behave with respect to disjoint union and the join with a
complete graph.

\begin{lemma}
\label{l:alpha}
$\displaystyle \MU(G) = \min_{l\in \mathbb{N}_0} \{ \overline{K_l}  \ast G \text{ is Hamiltonian }\}$
\end{lemma}

\begin{proof}
Since $\overline{K_l}\ast G$ is a subgraph of $K_l \ast G$,
a Hamiltonian cycle in  $\overline{K_l}\ast G$ would also be one in $K_l \ast G$.

Let $\MU(G)= a$.   Suppose   $C$ is a  Hamiltonian cycle in $K_a \ast G$ and write
$C$ as $v \sim P_1 \sim Q_1 \sim \ldots \sim P_s \sim Q_s \sim v$, where $v$ is a vertex
in $G$ and the paths $P_i \in G$ and $Q_i \in K_a$. If any
$Q_i$ contains 2 vertices or more, say $u$ and $w_1, \ldots, w_k$ with
$k \geq 1$, then we may simply remove all the vertices, except $u$,
and end up with a Hamiltonian graph on $K_{a-k}$. This
contradicts the minimality of $a=\MU(G)$.
Therefore, $C$ must not contain any paths of length greater than two in the subgraph
$K_{a}$, and  any Hamiltonian cycle on $K_a \ast G$ is  also a
Hamiltonian cycle on $\overline{K_a} \ast G$.
\end{proof}

\begin{lemma}
{\label{l:ami}}
$\MU(G) = \mu(G) -\h(G)$
\end{lemma}

\begin{proof}
If $G$ is Hamiltonian (including $P_1$ and $P_2$) then $\MU(G)=0$,
$\mu(G)= 1$ so the equality holds.  Suppose $G$ is non-Hamiltonian
with $\mu(G)= t$ and $t$-path covering  $P_1, \dots, P_t$.
Let $K_t$ have vertices $u_1, \dots, u_t$.  In the graph
$K_t*G$, there is a Hamiltonian cycle: $v_1\sim P_1\sim v_2\sim P_2
\sim\cdots\sim v_t \sim P_t\sim v_1$.  Thus $\MU(G) \leq t = \mu(G)$.

Let $\MU(G)= a$, so there is a Hamiltonian cycle in $K_a*G$.  Removing
the vertices of $K_a$ breaks the cycle into at most $a$ disjoint
paths covering $G$.  Thus $\mu(G) \leq \MU(G)$.
\end{proof}

\begin{lemma}
\label{l:disjoint}
$\mu(G \sqcup H) = \mu(G) + \mu(H)$ 
and $\MU(G \sqcup H) = \MU(G)+ \MU(H) 
+ \h(G)+\h(H)$.
\end{lemma}

\begin{proof}
A path covering of $G$  may be combined with a path covering of $H$ to create one for 
$G \sqcup H$.  Conversely, paths in a $t$-path covering of $G\sqcup H$ can be partitioned
into those contained in $G$ and those contained in $H$, giving a path covering of $G$ and
one of $H$. Consequently 
\[\mu(G\sqcup H) = \mu(G)+ \mu(H)\]

Since $G\sqcup H$ is not Hamiltonian we have   
\begin{align*}
\MU(G \sqcup H) &= \mu(G \sqcup H) + \h(G \sqcup H)\\
& = \mu(G) +   \mu(H) \\
& = \MU(G) + \h(G) + \MU(H) + \h(H)
\end{align*}

\end{proof}

\begin{lemma}{\label{l:star}}
For any graph $G$, 
\begin{align*}
\mu(K_s \ast G) &= \max \{1, \mu(G) -s \}\\
\MU(K_s \ast G) &= \max \{0, \MU(G) -s \}
\end{align*}
In particular, if $K_s*G$ is Hamiltonian then $\mu(K_s*G) = 1$ and $\MU(K_s*G) = 0$;
otherwise,  $\mu(K_s \ast G) =\mu(G) -s $ and $\MU(K_s \ast G) = \MU(G) -s$.
\end{lemma}

\begin{proof}
The formula for $\MU$ is immediate when $G$ is Hamiltonian since we have observed that this forces
$K_s*G$ to be Hamiltonian.  Otherwise, it follows from $K_r*(K_s*G)= K_{r+s}*G$: 
if $\MU(G)= a$, then $K_r*(K_s*G) $ is Hamiltonian if and only if $r+s\geq a$.

The formula for $\mu$ may be derived from the result for $\MU$ using Lemma~\ref{l:ami}.
We may also prove it directly.  Observe that it is enough to prove $\mu(K_1*G)= \max\{1,\mu(G)-1\}$.
Let $u$ be the vertex of $K_1$.  Let $\mu(G)= t$ and $P_1, \dots, P_t$ a $t$-path covering of $G$.
If $t=1$ then $u$ can be connected to the initial vertex of $P_1$ to create a 1-path covering of $K_1*G$.
For $t\geq 2$, the path $P_1\sim u \sim P_2$ along with $P_3, \dots, P_t$ gives a $(t-1)$-path covering of
$K_1*G$.  Thus for $t>1$,  $\mu(K_1*G) \leq t-1$.  
Suppose  $Q_1, \dots, Q_d$  were a  minimal $d$-path covering of $K_1*G$, with $u$ a vertex of $Q_1$. 
Removing $u$ gives at most a $(d+1)$-path covering of $G$.  Thus $\mu(K_1*G) +1 \geq t$.  
This shows $\mu(K_1*G)=\mu(G)-1$ for $\mu(G) \geq 2$.
\end{proof}

The main result of this section is the following two formulas for for
the $\mu$ and $\MU$ invariants for the disjoint union of graphs, and
the join with a complete graph.
\begin{proposition}{\label{p:lemmy}}
Let $\displaystyle \{G_j\}_{j=1}^m$ be graphs. 

$\displaystyle \mu\big(\bigsqcup_{j=1}^m G_j\big) =\sum_{j=1}^m\mu(G_j)$ and
$\displaystyle \MU\big(\bigsqcup_{j=1}^m G_j\big) =\sum_{j=1}^m\MU(G_j) + \sum_{j=1}^m\h(G_j)$.

Furthermore, $\displaystyle \MU \big( (\bigsqcup_{j=1}^m G_j)~\ast~K_r \big) =\max \big\{ 0 ,
\sum_{j=1}^m\MU(G_j) + \sum_{j=1}^m\h(G_j) -r \big\}$.
\end{proposition}

\begin{proof}
We proceed by induction. The base case $k=2$ is exactly Lemma~\ref{l:disjoint}.
Assume the formula holds for $k$ graphs we will prove it for $k+1$ graphs.
\begin{align*}
\mu\big(\bigsqcup_{j=1}^{k+1} G_j\big)   &= \mu\big(  (\bigsqcup_{j=1}^{k+1} G_j) \sqcup G_{k+1} \big) \\
&= \mu\big(\bigsqcup_{j=1}^{k} G_j\big)  + \mu\big( G_{k+1} \big) \\
 &= \sum_{j=1}^k \mu(G_j)  + \mu\big( G_{k+1} \big) \\
&=\sum_{j=1}^{k+1} \mu(G_j)
\end{align*}

By Lemma~\ref{l:ami} and the fact that disjoint graphs are not Hamiltonian, we have,

\begin{align*}
\MU\big(\bigsqcup_{j=1}^m G_j\big)    &= \mu\big(\bigsqcup_{j=1}^{m} G_j\big) + \h\big(\bigsqcup_{j=1}^{m} G_j\big) \\
&= \sum_{j=1}^m \mu(G_j) +0 \\
&=\sum_{j=1}^{m} (\MU(G_j) + \h(G_j)) \\
&=\sum_{j=1}^m\MU(G_j) + \sum_{j=1}^m\h(G_j)
\end{align*}

Therefore, we have by Lemma~\ref{l:star},

\begin{align*}
\MU\big( (\bigsqcup_{j=1}^{m} G_j) \ast K_r\big) &= \max \{ 0, \MU\big(\bigsqcup_{j=1}^{m} G_j\big)  - r \} \\
&= \max \{ 0, \sum_{j=1}^{m} \MU(G_j) + \sum_{j=1}^{m} \h(G_j) -r \} \\
\end{align*}

\end{proof}

The following lemma will be useful in the next section.  
To express it succintly we introduce the following Boolean condition.
For a graph $G$ and vertex $v\in G$, $T(v,G)$ is true if and only if $v$ is a terminal vertex in some minimal path
covering of $G$.

\begin{lemma}
Let $v \in G$ and $w \in H$.
\begin{align*}
\mu \Big( ( G \sqcup H)  + vw  \Big) = 
\begin{cases}
\mu(G\sqcup H) -1 & \text{ if } T(v,G) \text{ and } T(w,H)\\
\mu(G\sqcup H) & \text{ otherwise} 
\end{cases}
\end{align*}
\end{lemma}

\begin{proof}
Let $\mu(G) = c$, $\mu(H)= d$ and $\mu \Big( (G \sqcup H )+ vw\Big)  = t$. Clearly, $t \leq c+d$.

Let $R_1, \dots, R_t$ be a minimal path cover of $(G\sqcup H) + vw$.  If no $R_i$ contains
$vw$ then this is also a minimal path cover of $(G\sqcup H)$ so $t= c+d$.  
Suppose $R_1$ contains $vw$ and note that $R_1$  is the only path with
vertices in both $G$ and $H$.
Removing $vw$ gives two paths $P\subseteq G$ and $Q\subseteq H$. Paths $P$ and $Q$ along
with $R_2, \dots, R_t$ cover $G\sqcup H$, so $t+1 \geq c+d$.  Thus, $t$ can either be $c+d$ or $c+d -1$.

If $t = c+d -1$, then we have the minimal $(t+1)$-path covering $P,Q, R_2, \ldots, R_t$ of $G \sqcup H$, as above.
We note that $v$ must be a terminal point of $P$ and $w$ must be a 
terminal point of $Q$, by construction. This path covering may be partitioned into a $c$-path covering of $G$ containing $P$
and a $d$-path covering of $H$ containing $Q$. Thus, $T(v,G)$ and $T(w,G)$ hold.

Conversely, suppose  $T(u,G)$ and $T(w,H)$ both hold. 
Let $P_1, \dots, P_c$ be a minimal path of $G$ with $v$ a terminal vertex of $P_1$ and let $Q_1,
\dots, Q_d$ be a minimal path cover of $H$ with $w$ a terminal vertex of $Q_1$.    
The edge $vw$ knits $P_1$ and $Q_1$ into a single path and $P_1\sim Q_1, P_1, \dots, P_c, Q_1,
\dots, Q_d$ is a $c+d-1$ cover of $(G\sqcup H)+  vw$.  Consequently, $t \leq c+d-1$.

Thus,  $T(u,G)$ and $T(w,H)$ both hold if and only if $t = c+d -1$. Otherwise, $t= c+d$.
\end{proof}

\begin{corollary}
\label{c:disjointaddedge}
Let $v \in G$ and $w \in H$.
\begin{align*}
\MU \Big( ( G \sqcup H)  +  vw \Big) = 
\begin{cases}
\MU(G\sqcup H) -2 & \text{ if } G=H=K_1\\
\MU(G\sqcup H) -1 & \text{ if } T(v,G) \text{ and } T(w,H)\\
\MU(G\sqcup H) & \text{ Otherwise} 
\end{cases}
\end{align*}
\end{corollary}

\begin{proof}
Let $\delta = 1$ if $T(v,G)$ and  $T(w,H)$ are both true {\ch and $\delta = 0$ otherwise. Then}
\begin{align*}
\MU\Big ( (G \sqcup H)  +  vw \Big) &= 
\mu\Big ( (G \sqcup H)  +  vw \Big) - \h\Big ( (G \sqcup H)  +  vw \Big) \\ 
&= \mu( (G \sqcup H)  - \delta - \h\Big ( (G \sqcup H)  +  vw \Big) 
\end{align*}
The final term is $-1$ if and only if $G=H=K_1$.
\end{proof}


\section{Decomposing Maximal $t$-path traceable graphs}
\label{s:decomposition}
In this section we prove our main result, a maximal $t$-path traceable graph may be uniquely written as
the join of a complete graph and a disjoint union of graphs that are also maximal with respect to
traceability, but which are also either complete or have no universal vertex.
We work with the families of graphs $\calM_t$ for $t\geq0$ and $\calN_t$ for $t\geq 1$.

\begin{align*}
\mathscr{M}_t & := \{ G \vert   \MU(G) = t \text{ and  } \MU(G + e) < t, \forall e \in E(\overline{G}) \}\\
\mathscr{N}_t  &:= \{ G \in \mathscr{M}_t \vert G \text{ is connected and has no universal vertex } \}
\end{align*}

The set  $\mathscr{M}_0$ is the set of complete graphs.  The set $\mathscr{M}_1$ is the
set of graphs with a Hamiltonian path but no Hamiltonian cycle, that is, maximal non-Hamiltonian
graphs.
For $t > 1$, $\mathscr{M}_t$ is also the set of graphs $G$ such $\mu(G)=t$ and $\mu(G+e) = t-1$ for
any $e \in E(\overline{G})$.  We will call these {\em maximal   $t$-path traceable graphs}.  
A graph in $\calN_t$ will be called {\em trim}.

\begin{proposition}
\label{p:max1}
For $0 \leq s <t$, $G \in \calM_t$ if and only if $K_s*G \in \calM_{t-s}$.
\end{proposition}

\begin{proof}
We have $\MU(K_s*G) = \MU(G)-s$, so we just need to show that $K_s \ast G$ is
maximal if and only if $G$ is maximal.  The only edges that can be added to $K_s*G$ are those between
vertices of $G$, that is, $E(\overline{K_s*G}) = E(\overline{G})$.  For
such an edge $e$, 
\begin{align}
\label{e:max1}
\MU\Big( (K_s*G) +  e\Big) &= \MU\Big( K_s* (G +  e) \Big) \notag\\
&= \MU(G +  e) -s 
\end{align}
Consequently, $\MU(G+  e) = \MU(G)-1$ if and only if $\MU\Big( (K_s*G) +  e\Big) =
\MU(K_s*G)-1$.
\end{proof}

Note that the proposition is false for $s=t>0$ since $K_s*G$ will not be a complete graph and
$\calM_0$ is  the set of complete graphs.  The proof breaks down in \eqref{e:max1}.

\begin{proposition}
\label{p:cot}
Let $G \in \calM_c$ and $H \in \calM_d$.  The following are equivalent.
\begin{enumerate}
\item $G \sqcup H \in \calM_{c+d + \h(G) + \h(H)}$  
\item Each of $G$ and $H$ is either complete or has no universal vertex.
\end{enumerate}
\end{proposition}

\begin{proof}
We have already shown that 
$\MU(G \sqcup H) = c+d + \h(G) + \h(H)$.  We have to consider whether adding an edge to $G\sqcup H$
reduces the $\MU$-invariant.  There are three cases to consider, the extra edge may be in
$E(\overline{G})$ or $E(\overline{H})$ or it may join a vertex in $G$
to one in $ H$.
Since $G$ is maximal, adding an edge to $G$ is either
impossible, when $G$ is complete, or it reduces the $\MU$-invariant of $G$.  
This edge would also reduce the $\MU$-invariant of $G \sqcup H$ by Lemma~\ref{l:disjoint}.
The case for adding an edge of $H$ is the same.  Consider the edge
$vw$ for $v \in V(G)$ and $w \in V(H)$. 
By Corollary~\ref{c:disjointaddedge}  the $\MU$-invariant will drop if
and only if  $v$ is the terminal point of a path in a minimal path covering of $G$ and similarly
for $w$ in $H$, that is, $T(v,G)$ and  $T(w,H)$.
Clearly this holds for all vertices in a complete graph.  
The following lemma shows that $T(v,G)$ holds for $G\in \calM_c$ with $c>0$ if and only if $v$ is
not a universal vertex in $G$.   
Thus, in order for  $G \sqcup H$ to be maximal $G$ must either be complete, or be maximal itself,
and have no universal vertex, and similarly for $H$.
\end{proof}

As a key step before the main theorem, the next lemma shows that in a
maximal graph, each vertex is universal, or a terminal vertex in a minimal path
covering.

\begin{lemma}
Let $c \geq 1 $ and $G\in \calM_c$.
For any two non-adjacent vertices $v,w$ in $G$ there is a $c$-path covering of $G$ in which both $v$
and $w$ are terminal points of paths.  Moreover, a vertex $v\in G$ is a terminal point in some
$c$-path covering if and only if $v$ is not universal.
\end{lemma}

\begin{proof}
Suppose $c>1$ and let $v,w$ be non-adjacent in $G$.  Since $G$ is maximal $G +  vw$ has a $(c-1)$-path
covering, $P_1, \dots, P_{c-1}$.  The edge $vw$ must be contained in some $P_i$ because $G$ has
no $(c-1)$-path covering.  Removing that edge gives a $c$-path covering of $G$ with $v$ and $w$ as
terminal vertices.  The special case $c=1$ is  well known, adding the edge $vw$ gives a Hamiltonian
cycle, and removing it leaves a path with endpoints $v$ and $w$.
A consequence is that any non-universal vertex is the terminal point of some path in a $c$-path
covering.

Suppose  $P_1, \dots, P_c$ is a  $c$-path covering of $G\in \calM_c$ with $v$ a
terminal point of $P_i$.  Then $v$ is not adjacent to any of the terminal points of $P_j$ for $j\ne i$, for
otherwise two paths could be combined  into a single one.  In the case $c=1$, $v$ cannot be
adjacent to the other terminal point of $P_1$, otherwise $G$ would have a Hamiltonian cycle.
Consequently a universal vertex is not a terminal point in a $c$-path covering of $G$.
\end{proof}

\begin{theorem}
\label{t:decomposition}
For any $G \in \mathscr{M}_t$, $t >0$, $G$ may be uniquely decomposed  as
$\displaystyle K_s \ast ( G_1 \sqcup \ldots \sqcup G_r)$, where $s$ is
the number of universal vertices of $G$, and each $G_j$ is either complete
or $G_j\in \mathscr{N}_{t_j}$ for some  $t_j > 0$.  Furthermore
$\displaystyle t = \sum_{j=1}^r t_j + \sum_{j=1}^r\h(G_j)  - s$. 
\end{theorem}

\begin{proof}
Suppose $G \in \mathscr{M}_t$ and let $s$ be the number of universal
vertices of $G$.
Let $r$ be the number of  components in the graph obtained by removing
the universal vertices from  $G$, let $G_1, \dots G_r$ be the
components and let $\MU(G_j)= t_j$.

Proposition~\ref{p:lemmy} shows that 
$\displaystyle t = \sum_{j=1}^r t_j + \sum_{j=1}^r\h(G_j)  - r$.
By Proposition~\ref{p:max1}, we have that $G \in
\mathscr{M}_t$ if and only if $G_1 \sqcup \ldots \sqcup G_r \in
\mathscr{M}_{t+s}$.    Furthermore, each  $G_j$ must be in $\mathscr{M}_{t_j}$ for
otherwise we could 
. Without loss of generality if we add an edge $e$
to $G_1$, such that $\MU(G_1 +  e) < t_1$, then 

\begin{align*} 
\MU(G +  e) &=  \MU(G_1 +  e) + \sum_{j=2}^r t_j + \sum_{j=1}^r \h(G_j) -s \\
&<  \sum_{j=1}^r t_j + \sum_{j=1}^r \h(G_j) -s \\
&= t
\end{align*}

Now, we apply Proposition~\ref{p:cot}, so then $\displaystyle G_1 \sqcup \ldots \sqcup G_r \in \mathscr{M}_{t+s}$, where $\displaystyle t+s =\sum_{j=1}^r t_j + \sum_{j=1}^r\h(G_j)$ if and only if $G_j$ is either trim or complete. In other words, $G_j \in \mathscr{N}_{t_j}$ for $t_j > 0$ or $G_j \in \mathscr{M}_0$ for $t_j = 0$.
\end{proof}


\section{Trim maximal $t$-path traceable graphs}
\label{s:family}
 
Skupien ~\cite{SkupienMNH} discovered the first family of maximal non-Hamiltonian graphs, that is, graphs
in $\calM_1$.
These graphs are formed by taking the join with $K_r$ of the  disjoint union of $r+1$ complete graphs.
The smallest graph in $\calN_2$ is shown in Figure~\ref{f:triangle}.    Chv\'atal identified its join
with $K_1$ as the smallest maximal non-Hamilitonian graph that is not 1-tough, that is, not one of
the Skupien family.    Jamrozik, Kalinowski and Skupien~\cite{SkupienCat} generalized this
example to three different families.  Family $A1$ replaces each edge $u_iv_i$ with an arbitrary
complete graph containing $u_i$ and replaces the $K_3$ formed by the $u_i$ with an arbitrary
complete graph.  The result has four cliques, the first three disjoint from each other but each 
intersecting the fourth clique in a single vertex.  This graph is also in $\calN_2$ and its  join
with $K_1$ gives a maximal non-Hamiltonian graph.  Family $A2$ is formed by taking the join with
$K_2$ of the disjoint union of a complete graph and the graph in $\calN_2$ just described.
Theorem~\ref{t:decomposition} shows that the resulting graph is in $\calM_1$.
Family $A3$ is a modification of the $A1$ family based on the graph in Figure~\ref{f:A3}, which is
in $\calN_2$. 
Bullock, Frick, Singleton and van Aardt~\cite{Bullock} recognized that two  constructions of
Zelinka~\cite{Zelinka} gave maximal  non-traceable graphs, that is, elements of $\calM_2$.  
Zelinka's first construction is like the Skupien family: formed from $r+1$ complete graphs followed by the
join with $K_{r-1}$.  The Zelinka Type~II family contains graphs in $\calN_2$ that 
 are a significant  generalization of the graphs in Figures~\ref{f:triangle}~and~\ref{f:A3}.  
In this section we generalize this family further to get graphs in $\calN_t$ for arbitrary $t$.
Our starting point is the graph in Figure~\ref{f:whirligig}, which is in $\calN_3$.

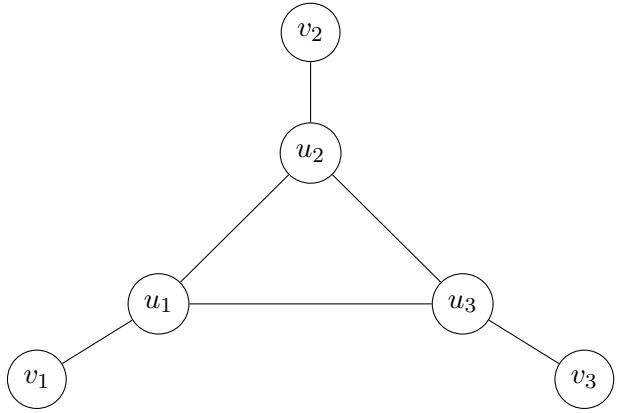
\begin{figure}
\begin{center}
\begin{tikzpicture}
  [scale=2,auto=left,every node/.style={circle,fill=white!20}]
  \node[draw, circle] (n1) at (1.2,1.5) {$v_1$};
  \node[draw, circle] (n2) at (3,3.8)  {$v_2$};
  \node[draw, circle] (n3) at (4.8,1.5)  {$v_3$};
  \node[draw, circle] (n4) at (2,2) {$u_1$};
  \node[draw, circle] (n5) at (3,3)  {$u_2$};
  \node[draw, circle] (n6) at (4,2)  {$u_3$};

\foreach \from/\to in {n1/n4,n4/n5,n5/n6,n6/n4,n2/n5,n3/n6}
  \draw (\from) -- (\to);

\end{tikzpicture}
\end{center}

\caption{Smallest graph in $\calN_2$}
\label{f:triangle}
\end{figure}

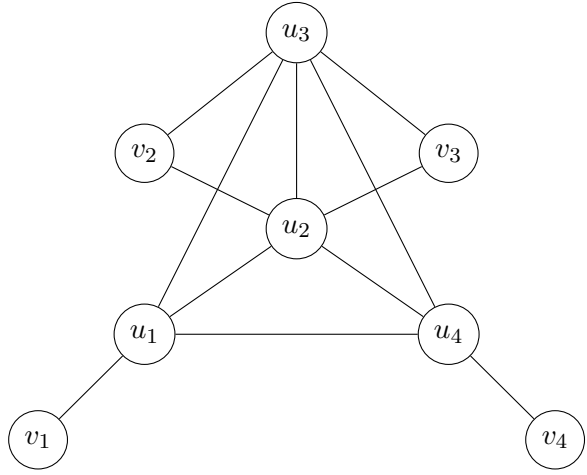
\begin{figure}
\begin{center}
\begin{tikzpicture}
  [scale=2,auto=left,every node/.style={circle,fill=white!20}]
  \node[draw, circle] (n1) at (1.3,1.6) {$v_1$};
  \node[draw, circle] (n2) at (2,2.3)  {$u_1$};
  \node[draw, circle] (n3) at (2,3.5)  {$v_2$};
  \node[draw, circle] (n4) at (3,3) {$u_2$};
  \node[draw, circle] (n5) at (3,4.3)  {$u_3$};
  \node[draw, circle] (n6) at (4,2.3)  {$u_4$};
  \node[draw, circle] (n7) at (4,3.5)  {$v_3$};
  \node[draw, circle] (n8) at (4.7,1.6)  {$v_4$};


\foreach \from/\to in {n2/n4,n2/n5,n2/n6,n4/n5,n4/n6,n5/n6,
n1/n2,n3/n4,n3/n5,n7/n4,n7/n5,n8/n6}
  \draw (\from) -- (\to);

\end{tikzpicture}
\end{center}

\caption{The join of this graph with $K_1$ is the smallest graph in
  the $A3$ family.}
\label{f:A3}
\end{figure}

\begin{example}
\label{ex:whirligig}
Consider  $K_m$ with $m\geq 2t-1$ and vertices $u_1, \dots, u_m$.
Let $G$ be the graph containing $K_m$ along with vertices $v_1, \dots, v_{2t-1}$ and edges $u_iv_i$.
The case with $t=3$ and $m=5=2t-1$ is  Figure~\ref{f:whirligig}.
We claim $G \in \calN_t$.

One can readily check that this graph is $t$-path covered using 
$v_{2i-1} \sim u_{2i-1} \sim u_{2i} \sim v_{2i}$ for $i= 1, \dots, t-1$ and $v_{2t-1}\sim u_{2t-1}
\sim u_{2t} \sim \cdots \sim u_m$.  We check that $G$ is maximal.   By the symmetry of the
graph,  we need only consider the addition of the edge $v_1u_m$ and $v_1u_2$.  In either case, the
last and the first paths listed above may be combined into one, either 
\begin{align*}
&v_{2t-1} \sim u_{2t-1} \sim \cdots \sim u_m \sim  v_1 \sim u_1 \sim u_2 \sim v_2, \quad \text{ or } \quad \\
&v_{2t-1} \sim u_{2t-1} \sim \cdots \sim u_m \sim u_1  \sim v_1 \sim u_2 \sim v_2
\end{align*}
Thus, adding an edge creates a $(t-1)$-path covered graph, proving maximality.
\end{example}

\begin{figure}
\begin{center}
\rot{270}{
\begin{tikzpicture}
  [scale=1.5,auto=left,every node/.style={circle,fill=white!20}]
  \node[draw, circle] (n1) at (2.2,3) {\rot{90}{$u_1$}};
  \node[draw, circle] (n2) at (3,4)  {\rot{90}{$u_2$}};
  \node[draw, circle] (n3) at (4,3.7)  {\rot{90}{$u_3$}};
  \node[draw, circle] (n4) at (4,2.3) {\rot{90}{$u_4$}};
  \node[draw, circle] (n5) at (3,2)  {\rot{90}{$u_5$}};
  \node[draw, circle] (n6) at (1,3)  {\rot{90}{$v_1$}};
  \node[draw, circle] (n7) at (3,5) {\rot{90}{$v_2$}};
  \node[draw, circle] (n8) at (5,4.5)  {\rot{90}{$v_3$}};
  \node[draw, circle] (n9) at (5,1.5)  {\rot{90}{$v_4$}};
  \node[draw, circle] (n10) at (3,1) {\rot{90}{$v_5$}};

\foreach \from/\to in {n1/n2,n1/n3,n1/n4,n1/n5,n2/n4,n2/n5,n3/n4,n3/n5,n2/n3,n4/n5,
n1/n6,n2/n7,n3/n8,n4/n9,n5/n10}
  \draw (\from) -- (\to);

\end{tikzpicture}
}
\end{center}
\caption{Whirligig in $\calN_3$.} 
\label{f:whirligig}
\end{figure}
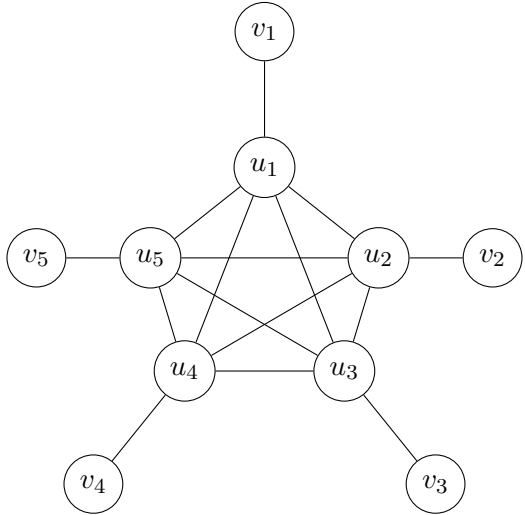

The next proposition shows that the previous example is the only way to have a trim maximal $t$-path
covered graph with $2t-1$ degree-one vertices.  We start with a
technical  lemma.H

\begin{lemma}
\label{l:degone}
Let $G$ be a connected graph and let $u_1, v_1, v_2, v_3 \in G$ with $\deg(v_i) = 1$, and $u $ adjacent
to $v_1$ and $v_2$ but not $v_3$.  
Then $\mu(G) =  \mu(G + uv_3 )$.
\end{lemma}

\begin{proof}
Let $P_1, \dots, P_r$ be a minimal path covering of $G +uv_3$; 
it is enough to show that there are $r$-paths covering $G$.  
If the covering doesn't include $uv_3$, then $P_1,\dots, P_r$ also give a minimal
path covering of $G$ establishing the claim of the lemma.  
Otherwise, suppose $uv_3$ is  an edge of $P_1$.
We consider two cases.

Suppose $P_1$ contains the edge $uv_1$ (or similarly $uv_2$). 
Then  $P_1$ has $v_1$ as a terminal point and one of the other paths, say $P_2$ must be a
length-$0$ path containing simply     $v_2$.  Let $Q$ be obtained by removing $uv_1$ and $uv_3$ from
$P_1$. Then  $v_1\sim u\sim v_2, Q, P_3, \dots, P_r$, gives an $r$-path covering of $G$

Suppose $P_1$ contains neither  $uv_1$ nor  $uv_2$. 
Then each of $v_1$ and $v_2$  must be on a length-$0$ path in the covering, say $P_2$ and $P_3$ are
these paths.   Furthermore $u$ must not be a terminal point of $P_1$, for, if were, the path could
be extended to include $v_1$ or $v_2$, reducing the number of paths required to cover $G$.
Removing $u$ from $P_1$ yields  two paths, $Q_1, Q_2$.
Then $v_1\sim u\sim v_2, Q_1, Q_2, P_4, \dots, P_r$ gives an $r$-path cover of $G$.  This proves the
lemma.

\end{proof}

\begin{proposition}
Let $G\in \calN_t$.  The number of degree-one vertices in $G$ is at most $2t-1$.
This occurs if and only if the $2t-1 $ vertices of degree-one have
distinct neighbors and removing the degree-one vertices leaves a
complete graph.
\end{proposition}

\begin{proof}
Each degree-one vertex must be a terminal point in a path covering.
So any graph $G$ covered by $t$ paths can have at most $2t$ degree-one
vertices.  
Aside from the case $t=1$ and $G=K_2$,  we can see that a graph with $2t$ degree-one vertices cannot be
maximal $t$-path traceable as follows.  It is easy to check that a $2t$ star is not 
$t$-path traceable (it is also not trim). A $t$-path traceable graph with $2t$ degree-one vertices
must therefore have an interior vertex $w$ that is not connected to one of the degree-one vertices $v$.
Such a graph is not maximal  because the edge $vw$ can 
be added leaving $2t-1$ degree-one vertices. This graph cannot be $(t-1)$-path covered.

Suppose that $G\in \calN_t$ with $2t-1$ degree-one vertices, $v_1,
\dots, v_{2t-1}$. 
Lemma~\ref{l:degone} shows that no two of the $v_i$  can be adjacent
to the same vertex, for that would violate maximality of $G$.
So, the $v_i$ have distinct neighbors.
Furthermore, all the nodes except the $v_i$ can be connected to each
other and a path covering will still require at least $t$ paths since there remain $2t-1$ degree-one
vertices. 
This proves the necessity of the structure claimed in the proposition.  The previous example showed
that the graph is indeed in $\calN_t$.

\end{proof}

We can now generalize the Zelinka family.

\begin{construction}
\label{C:whirligig}
Let $U_0, U_{1}, \dotsm U_{2t-1}$ be disjoint sets and $\displaystyle U= \bigsqcup_{i=0}^{2t-1} U_i$.
Let $m_i = \abs{U_i} $ and assume that for $i>0$ the $U_i$ are
non-empty, so $m_i >0$. For $i=1,\dots, 2t-1$ (but not $i=0$) and $j = 1, \dots, m_i$, let $V_{ij}$ be disjoint
from each other and from $U$.
Form the graph with vertex set $\displaystyle U \sqcup \Big( \bigsqcup_{i=1}^{2t-1} \big(\bigsqcup_{j=1}^{m_i}
V_{ij}\big)\Big)$ and edges $uu'$ for $u,u' \in U$ and $uv$ for any 
$u \in U_i$  and $v \in V_{ij}$ with  $i= 1,\dots, 2t-1$ and $j=1,\dots,m_i$.
The cliques of this graph are $K_U$ and $K_{U_i\sqcup V_{ij}}$ for each  $i= 1,\dots, 2t-1$ and $j=1,\dots,m_i$.
\end{construction}

The graph in Figure~\ref{f:A3} has $m_0=0$, $m_1=m_2=1$ and $m_3=2$,
and the graph in Figure~\ref{f:general} indicates the general construction.

\begin{figure}
\begin{tikzpicture}
  [scale=1.6,auto=left,every node/.style={circle,fill=white!20}]
  \node (n1) at (1.3,3) {$U_0$};
  \node (n2) at (2.3,4)  {$U_1$};
  \node (n3) at (2.3,2)  {$U_2$};
  \node (n4) at (3.5,3) {$\iddots$};
  \node (n6) at (4.7,4)  {$U_{2t-3}$};
   \node (n7) at (4.7,2) {$U_{2t-2}$};
  \node (n8) at (5.7,3)  {$U_{2t-1}$};

 \node (n9) at (1,5)  {$V_{1,1}$};
  \node (n10) at (2,5) {$\ldots$};
  \node (n11) at (3,5)  {$V_{1,m_1}$};

  \node (n12) at (4,5)  {$V_{2t-3,1}$};
    \node (n13) at (5,5) {$\ldots$};
  \node (n14) at (6,5)  {$V_{2t-3,m_{2t-3}}$};

  \node (n15) at (1,1)  {$V_{2,1}$};
  \node (n16) at (2,1) {$\ldots$};
  \node (n17) at (3,1)  {$V_{2,m_2}$};

  \node (n18) at (4,1)  {$V_{2t-2,1}$};
    \node (n19) at (5,1) {$\ldots$};
  \node (n20) at (6,1)  {$V_{2t-2,m_{2t-2}}$};

  \node (n21) at (7,4)  {$V_{2t-1,1}$};
  \node (n22) at (7,3) {$\vdots$};
  \node (n23) at (7,2)  {$V_{2t-1,m_{2t-1}}$};

  \foreach \from/\to in {n2/n9,n2/n10,n2/n11,
n3/n15,n3/n16,n3/n17,
n6/n12,n6/n13,n6/n14,
n7/n18,n7/n19,n7/n20,
n8/n21,n8/n22,n8/n23}
    \draw (\from) -- (\to);

\draw (1.3,3.7) to[out=-45,in=45] (1.3,2.3);
\draw (1.65,4.05) to[out=-70,in=180] (2.3,3.5) to[out=0,in=-60] (2.9,4.45);
\draw (1.65,1.95) to[out=70,in=180] (2.3,2.5) to[out=0,in=60] (2.9,1.55);
\draw (5.7,3.7) to[out=200,in=160] (5.7,2.3);
\draw (5.35,4.05) to[out=250,in=0] (4.7,3.5) to[out=180,in=240] (4.1,4.45);
\draw (5.35,1.95) to[out=110,in=0] (4.7,2.5) to[out=180,in=120] (4.1,1.55);

\draw (3.5,3) ellipse (2.5cm and 1.5cm);


\end{tikzpicture}

\caption{Generalization of the Whirligig, $W$} 
\label{f:general}
\end{figure}
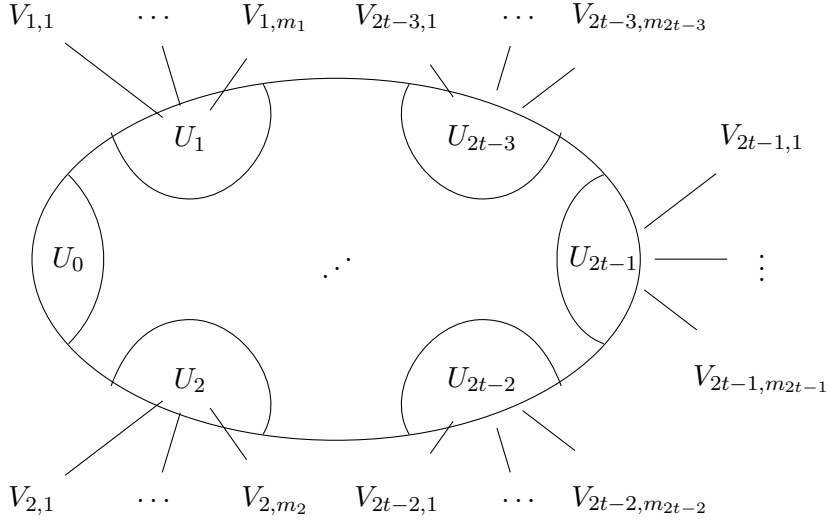

\begin{theorem}
The graph $W$ in Construction~\ref{C:whirligig}  is a trim, maximal $t$-path traceable graph.
\end{theorem}

\begin{proof}
We must show that $W$ is $t$-path covered and not $(t-1)$-path covered, and that the addition of
any edge yields a  $(t-1)$-path covered graph.  The argument is analogous to the one in
Example~\ref{ex:whirligig}.

Let $R$ be a Hamiltonian path in $U_0$.
For each $i=1,\dots,2t-1$ and $j=1,\dots,m_i$ let $Q_{ij}$ be a Hamiltonian path in $K_{V_{ij}}$.  
Let $P_i$ be the path
\[
P_i: Q_{i1} \sim u_{i1} \sim \cdots \sim Q_{im_i} \sim u_{im_i}
\]
and let $\rv{P_i}$  be the reversal of $P_i$.  

Since there is an edge $u_{im_i}u_{jm_j}$ there is a path $P_i\sim \rv{P}_j$ for any $i\ne
j\in \set{1,\dots, 2t-1}$.  Therefore the graph $W$ has a $t$-path covering $P_{2i-1} \sim
\rv{P}_{2i}$ for $i=1,\dots, (t-1)$ ,  
along with $P_{2t-1} \sim R$.  
We leave to the reader the argument that there is no $(t-1)$-path cover.

To show $W$ is maximal we show that after adding an edge $e$, 
we can join two paths in the $t$-path cover above, with a bit of rearrangement.
There are three types of edges to consider, the edge $e$ might join  $V_{ij}$ to
$U_{i'}$ for $i \ne i'$; or  $V_{ij}$ to $V_{ij'}$ for $j\ne j'$; or $V_{ij}$ to $V_{i'j'}$ for
$i\ne i'$.  
Because of the symmetry of $W$, we may assume $i=1$ and $j=1$ and that the vertex chosen
from $V_{ij}$ is the initial vertex of $Q_{ij}$.  
Other simplifications due to symmetry will be evident in what follows.
  
In the first case there are two subcases---determined by $i'\geq 2t$ or  not---and after permutation, 
we may consider the edge $e$ from the initial vertex of $Q_{11}$ to the terminal vertex of $R$, or
to the terminal vertex of $P_{2t-1}$.  
We can then join two paths in the $t$-path cover:  
either $P_{2t-1} \sim R\stackrel{ e}{\sim } P_1\sim \rv{P}_2$ 
or $P_2 \sim  \rv{P}_1\stackrel{e}{ \sim }P_{2t-1}\sim R$.

Suppose next that we join the initial vertex of $Q_{11}$ with the terminal vertex of $Q_{12}$.
We then rearrange $P_1$ and join  two path in the $t$-path cover to get
\[P_{2t-1}\sim R \sim u_{11} \sim Q_{11}\stackrel{e}{\sim} Q_{12} \sim u_{12}\sim \cdots \sim Q_{1m_1}\sim
u_{1m_1}\sim \rv{P}_2 \]

Finally, suppose that we join the initial vertex of $Q_{11}$ with the initial vertex of
$Q_{2t-1,1}$.
Then we rearrange to $\rv{R}\sim\rv{P}_{2t-1}\stackrel{e}{\sim}P_1\sim\rv{P}_2$.

\end{proof}




\begin{thebibliography}{1}

\bibitem{Bullock} F.~Bullock, M.~Frick, J.~Singleton,~S.~van~Aardt,~K.~Mynhardt, \emph{Maximal Nontraceable Graphs with Toughness less than One}, Electronic~Journal~of~Combinatorics \textbf{18} (2008), \#R18.


\bibitem{SkupienCat} J.~Jamrozik,~R.~Kalinowski,~Z.~Skupien, \emph{A Catalogue of Small Maximal Nonhamiltonian Graphs}, Discrete~Mathematics\textbf{39} (1982), 229-234.

\bibitem{SkupienMNHT} A.~Marcyzk,~Z.~Skupien, \emph{Maximum nonhamiltonian tough graphs,} Discrete~Mathematics \textbf{96} (1991), 213-220.

\bibitem{Noorvash} S.~Noorvash, \emph{Covering the vertices of a graph by vertex-disjoint paths,} Pacific~Journal~of~Mathematics \textbf{58} (1975), 159-168.

\bibitem{Ore} O.~Ore, \emph{Arc Coverings of graphs,} Ann.~Mat.~Ser.~IV \textbf{55} (1961), 315-321.


\bibitem{SkupienMNH}
 ~Z.~Skupien, \emph{On Maximum non-Hamiltonian graphs}, Rostock. Math. Kolloq. \textbf{11} (1979), 97-106.


\bibitem{Zelinka} B.~Zelinka, \emph{Graphs maximal with respect to absence of Hamiltonian Paths,} Discussiones~Mathematicae,~Graph~Theory \textbf{18} (1998), 205-208.


\end{thebibliography}
\end{document}